\documentclass{article}

\usepackage{amsmath}
\usepackage{amsthm}
\usepackage{amssymb}

\usepackage[utf8]{inputenc}
\usepackage[T1]{fontenc}
\usepackage{lmodern}

\newtheorem{thm}{Theorem}[section]
\newtheorem{lemma}[thm]{Lemma}

\newtheorem{definition}[thm]{Definition}

\begin{document}

\title{The Probability Distribution for Draws Until First Success Without Replacement}

\author{John Ahlgren $<$ahlgren@ee.cityu.edu.hk$>$}

\maketitle

\begin{abstract}
We consider the urn setting with two different objects, ``good'' and ``bad'', and analyze the number of draws without replacement until a good object is picked. Although the expected number of draws for this setting is a standard
textbook exercise, we compute the variance, and show that this distribution converges to the geometric distribution.
\end{abstract}

\section{Introduction}
Consider the setting where we draw from an urn containing two types of objects, ``good'' and ``bad''.
There are in total $N$ objects, of which $K$ are good.

Most instances of this setting are well known \cite{probability-theory-book}: if we wish to obtain the number of good draws among $n$ many, use the binomial distribution for draws with replacement. Without replacement, we instead use the hypergeometric distribution.
If, on the other hand, we wish to consider the number of draws until we draw a good object, we use the geometric distribution for draws with replacement. But what about draws without replacement? 
The situation is depicted in Table~\ref{tab:classification}.

It is not too hard to show that the expected number of draws in this case is $(N+1)/(K+1)$, but what about the variance?
Also, intuitively, such a distribution should be approximated by a geometric distribution if $N$ is much larger than $K$, as
the replacements will matter less.

\begin{table}
\centering{}
\begin{tabular}{|l|l|l|} 
\hline
Setting & Counting Solutions & Until First Solution \\
\hline
Replacements & Binomial & Geometric \\
No Replacements & Hypergeometric & ? \\
\hline
\end{tabular}
\caption{Classification of Distributions.}
\label{tab:classification}
\end{table}

In Section~\ref{sec:distribution}, we present the relevant probability mass function, cumulative function distribution,
expected value, variance, and a convergence result.
Longer proofs are given in the appendices. Section~\ref{sec:conclusions} concludes the paper.

\section{The Probability Distribution} \label{sec:distribution}
Consider an urn containing $N$ objects, of which $K$ are considered ``good'' and the remaining ``bad''.
The probability of picking a good object on the first trial $n=1$ is clearly $P(1) = \frac{K}{N}$.
For the second pick, we must have failed once and then succeeded: $P(2) = \frac{N-K}{N} \frac{K}{N-1}$.
To obtain the general form, we begin with an expression for the first $n$ failures.
\begin{thm} \label{theorem:failn}
The probability that our first $n$ picks are bad candidates (assuming independence) is 
\begin{equation} \label{eq:fail-ratio}
\frac{(N-K)(N-K-1)\ldots (N-K-n+1)}{N(N-1)\ldots (N-n+1)}
\end{equation}
which is well defined for $n=1,2,\ldots,(N-K)$.
\end{thm}
\begin{proof}
At step $k$, $k-1$ bad picks have been removed.
Thus there are $N-K-(k-1)$ bad picks left and a total of $N-(k-1)$ picks left.
The probability of failure is therefore $(N-K-k+1)/(N-k+1)$.
There are $n$ failure steps, $k=1,2,\ldots,n$, and since each step is independent, the result is obtained.
Note that we can only fail at most $N-K$ times, as that is the total number of bad picks.
\end{proof}

By expressing Equation~\ref{eq:fail-ratio} as a ratio of binoms, we define the Fail-function.
\begin{definition}
The Fail-function is defined as
\begin{equation}
Fail(n) = \frac{\binom{N-n}{K}}{\binom{N}{K}}
\end{equation}
for any integers $K$, $N$, and $n$.
\end{definition}
Due to Theorem~\ref{theorem:failn}, this corresponds to the probability of failure for $n=1,2,\ldots,(N-K)$.
Thus the Fail-function extends expression~\ref{eq:fail-ratio} so that it is $0$ when $n > N-K$ and $1$ for $n=0$.

Since the probability of success at $n$ is simply $n-1$ failures followed by a success, we can derive a closed
form for the probability mass function.
\begin{thm}[Probability Mass Function] \label{theorem:pmf}
The probability mass function (pmf) is:
\begin{equation} \label{eq:pmf}
P(n) = Fail(n-1) \frac{K}{N-n+1} = 
\frac{\binom{N-n+1}{K}}{\binom{N}{K}} \frac{K}{N-n+1}.
\end{equation}
\end{thm}
\begin{proof}
We must fail $n-1$ times and then succeed on our $n$th time. Using Theorem~\ref{theorem:failn} and noting that after having failed $n-1$ times
there are $N-(n-1) = N-n+1$ candidates left, we obtain the expression for $n \geq 2$.
For $n=1$, $\mathrm{Fail}(0) = 1$ and we obtain the correct probability $K/N$.
\end{proof}

To obtain a closed form for the cumulative distribution function, we take another approach.
\begin{thm}[Cumulative Distribution Function] \label{theorem:cdf}
The cumulative distribution function (cdf) is:
\begin{equation} \label{eq:cdf}
F(n) = \sum_{k=1}^n P(k) = 1 - Fail(n) = 1 - \frac{\binom{N-n}{K}}{\binom{N}{K}}
\end{equation}
\end{thm}
\begin{proof}
The cmf is the probability of succeeding at any of the first $n$ picks, which is equivalent to not failing
on any of the $n$ picks.
\end{proof}

Now that we have obtained the pmf and cdf, we are ready to present all results.
For the remainder of this paper, let $X$ be a random variable with pmf as in Equation~\ref{eq:pmf}, that is,
$X$ is the number of picks until we find a good enough solution.
\begin{thm}[Expected Value] \label{theorem:expected-value}
\begin{equation}
\mathrm{E}(X) = \frac{N+1}{K+1}.
\end{equation}
\end{thm}
\begin{proof}
See Appendix~\ref{sec:expected-proof}.
\end{proof}

The median does not have a nice closed form solution, but we can express it as a maximization problem using
the cdf.
\begin{thm}[Median] \label{theorem:median}
The median is the smallest integer $m$ such that
\begin{equation}
\binom{N-m}{K} \leq \frac{1}{2} \binom{N}{K}.
\end{equation}
\end{thm}
\begin{proof}
Immediate from Equation~\ref{eq:cdf}.
\end{proof}

The mode gives the most likely outcome.
\begin{thm}[Mode] \label{theorem:mode}
$\mathrm{MODE}(X) = 1$ if $K > 1$.
If $K = 1$, $P(n)$ has fixed probability $1/N$ for $n=1,2,\ldots,(N-K+1)$.
\end{thm}
\begin{proof}
See Appendix~\ref{sec:mode-proof}.
\end{proof}
When $K > 1$, both picking with and without replacement are most likely to find a good solution in the first pick.
Interestingly, for $K=1$, we obtain a uniform distribution with $P(n) = 1/N$.

The variance gives the predictability of performance.
\begin{thm}[Variance] \label{theorem:variance}
\begin{equation}
\mathrm{VAR}(X) = \frac{K(N-K)(N+1)}{(K+2)(K+1)^2}
\end{equation}
\end{thm}
\begin{proof}
See Appendix~\ref{sec:variance-proof}.
\end{proof}
In the special case where $K=1$, we have the variance for a uniform distribution, $(N^2-1)/12$.

Finally, we show that drawing without replacement can in some sense be approximated by its replacement counterpart as $N$
grows and the proportion of good solutions is held fixed.
\begin{thm} \label{theorem:distconverge}
Let $K/N$ be fixed. Let $Y$ be a random variable modeling draws with replacement.
$X \overset{d}{\longrightarrow} Y$ as $N \longrightarrow \infty$.
\end{thm}
\begin{proof}
See Appendix~\ref{sec:convergence-proof}.
\end{proof}

\section{Conclusion} \label{sec:conclusions}
We have presented a probability distribution for number of draws until a first success is achieved among $K$ good
and $N-K$ bad objects. Although the expected value of such a distribution is already known, we have also
presented its variance and shown that the distribution converges towards a geometric distribution when $N$ grows and
the proportion of good solutions $K/N$ is held fixed.

\appendix
\section{Proof of Expected Value} \label{sec:expected-proof}
First we start with a lemma.
\begin{lemma} \label{lemma:binom}
\begin{displaymath}
\sum_{j=K}^{N} \binom{j}{K} = \binom{N+1}{K+1}.
\end{displaymath}
\end{lemma}
\begin{proof}
By induction over $N$. For $N=K$, the equality is trivial.
Now assume it is valid for $N=P$. Then, for $N=P+1$, using Pascal's triangle,
\begin{equation}
\sum_{j=K}^{P+1} \binom{j}{K} = \binom{P+1}{K+1} + \binom{P+1}{K} = \binom{P+2}{K}.
\end{equation}
\end{proof}

Using Lemma~\ref{lemma:binom}, we can obtain the expected value in closed form.
\begin{proof}[Proof of Expected Value]
First, reindex the sum and move out all factors independent of the summation index:
\begin{displaymath}
\sum_{n=1}^{N-K+1} nP(n) = \frac{K}{\binom{N}{K}} \sum_{j=K}^{N} \frac{N+1-j}{j} \binom{j}{K}.
\end{displaymath}
Expanding and using $\binom{a}{b} = \frac{a}{b} \binom{a-1}{b-1}$, we get two sums:
\begin{displaymath}
\frac{1}{\binom{N}{K}} \left( (N+1) \sum_{j=K}^{N} \binom{j-1}{K-1} - K \sum_{j=K}^{N} \binom{j}{K} \right)
\end{displaymath}
which, by Lemma~\ref{lemma:binom}, simplifies to
\begin{displaymath}
\frac{1}{\binom{N}{K}} \left( (N+1) \binom{N}{K} - K \binom{N+1}{K+1} \right).
\end{displaymath}
Simplifying, this yields the result.
\end{proof}

\section{Proof of Mode} \label{sec:mode-proof}
To prove Theorem~\ref{theorem:mode}, we prove that the probability mass function $P$ is
a decreasing function when $K > 1$ and constant when $K=1$.
\begin{thm}[Decreasing PMF]
Let $n=1,2,\ldots,(N-K+1)$.
For $K > 1$, $P(n) > P(n+1)$. 
For $K=1$, $P(n+1) = P(n) = 1/N$.
\end{thm}
\begin{proof}
Using Equation~\ref{eq:pmf}, for $n=1,2,\ldots,(N-K+1)$,
\begin{equation}
\frac{P(n)}{P(n+1)} = \frac{N-n}{N-n+1-K}
\end{equation}
from which the result follows.
\end{proof}

\section{Proof of Variance} \label{sec:variance-proof}
Before we prove the expression for variance, we will need two new lemmas.
The first is a generalization of Lemma~\ref{lemma:binom}.
\begin{lemma} \label{lemma:binom1p}
For $x \geq K$,
\begin{displaymath}
\sum_{j=x}^{N} \binom{j}{K} = \binom{N+1}{K+1} - \binom{x}{K+1}.
\end{displaymath}
\end{lemma}
\begin{proof}
By writing
\begin{displaymath}
\sum_{j=x}^{N} \binom{j}{K} = \sum_{j=K}^{N} \binom{j}{K} - \sum_{j=K}^{x-1} \binom{j}{K}
\end{displaymath}
and then applying Lemma~\ref{lemma:binom} to both sums.
\end{proof}

\begin{lemma} \label{lemma:binom2}
\begin{displaymath}
\sum_{j=K}^{N} j\binom{j}{K} = N\binom{N+1}{K+1} - \binom{N+1}{K+2}.
\end{displaymath}
\end{lemma}
\begin{proof}
The result is obtained by swapping rows and columns in the summation.
First, consider the summation as a triangle:
\begin{displaymath}
\sum_{j=K}^N j \binom{j}{K} = \sum_{j=K}^N \sum_{\# j} \binom{j}{K}.
\end{displaymath}
Transposing the sum, we get
\begin{displaymath}
K \sum_{j=K}^N  \binom{j}{K} + \sum_{i=K+1}^N \sum_{j=i}^N \binom{i}{K}.
\end{displaymath}
The first sum can be solved using Lemma~\ref{lemma:binom},
and the second using Lemma~\ref{lemma:binom1p}, whereby we obtain
\begin{displaymath}
K \binom{N+1}{K+1} + \sum_{i=K+1}^N \left( \binom{N+1}{K+1} - \binom{i}{K+1} \right).
\end{displaymath}
Using Lemma~\ref{lemma:binom} once more, we obtain the result.
\end{proof}

\begin{proof}[Proof of Variance]
As with expected value, reindex the sum:
\begin{displaymath}
\sum_{n=1}^{N-K+1} n^2 P(n) = \frac{1}{\binom{N}{K}} \sum_{j=K}^{N} (N+1-j)^2 \binom{j-1}{K-1}
\end{displaymath}
which after expanding the square and using $\binom{a}{b} = \frac{a}{b} \binom{a-1}{b-1}$ gives:
\begin{multline}
\frac{1}{\binom{N}{K}} \bigg( (N+1)^2 \sum_{j=K}^{N} \binom{j-1}{K-1} + K \sum_{j=K}^{N} j \binom{j-1}{K-1} \\ - 2K(N+1) \sum_{j=K}^{N} j \binom{j}{K} \bigg).
\end{multline}
After using Lemma~\ref{lemma:binom} and \ref{lemma:binom2}, and simplifying, we obtain the stated expression.
\end{proof}

\section{Proof of Convergence} \label{sec:convergence-proof}
\begin{proof}[Convergence of Distribution]
Set $p=K/N$ and $q=(N-K)/N$. Clearly $K\longrightarrow \infty$ when $N \longrightarrow \infty$.
Note that, for any $n$, we have
\begin{displaymath}
\frac{N-K-n+1}{N-n+1} = 1-\left( \frac{1}{p} + \frac{1-n}{K} \right)^{-1} \longrightarrow q \text{ as } N \longrightarrow \infty.
\end{displaymath}
This implies that
\begin{displaymath}
Fail(n) = \prod_{k=1}^{n} \frac{N-K-k+1}{N-k+1} \longrightarrow q^{n}
\end{displaymath}
so that
\begin{multline}
P(n) = Fail(n-1) \left( 1 - \left(\frac{N-K-n+1}{N-n+1}\right) \right) \\
\longrightarrow q^{n-1} (1-q) \text{ as } N \longrightarrow \infty.
\end{multline}
\end{proof}

\section*{Acknowledgements}
The work described in this paper was supported by a grant from the Research Grants Council of the Hong Kong Special Administrative Region, China [Project No. CityU 125313].

\bibliographystyle{plain}

\bibliography{distributions}

\end{document}